\documentclass[11pt]{amsart}
\usepackage[colorlinks, citecolor=blue, dvipdfm,
pagebackref]{hyperref}
\usepackage{graphicx}
\usepackage[all]{xy}
\usepackage{extarrows}

\usepackage{tikz}

\newtheorem{theorem}{Theorem}[section]
\newtheorem{lemma}[theorem]{Lemma}

\theoremstyle{definition}
\newtheorem{definition}[theorem]{Definition}

\newtheorem{example}[theorem]{Example}

\newtheorem{proposition}[theorem]{Proposition}
\newtheorem{corollary}[theorem]{Corollary}
\newtheorem{remark}[theorem]{Remark}

\theoremstyle{remark}

\newcommand{\be}{\begin{equation}}
\newcommand{\ee}{\end{equation}}

\numberwithin{equation}{section}

\begin{document}
\title[Hamitonian $S^1$-action, invariant hypersurface and $\mathbb{P}^n$]{Hamiltonian circle action, invariant hypersurface and the complex projective space}
\author{Ping Li}
\address{School of Mathematical Sciences, Fudan University, Shanghai 200433, China}

\email{pinglimath@fudan.edu.cn\\
pinglimath@gmail.com}
\thanks{The author was partially supported by the National
Natural Science Foundation of China (Grant No. 12371066).}



\dedicatory{Dedicated to Professor Haibao Duan on the occasion of his 70th birthday.}

\begin{abstract}
Let $M$ be a $2n$-dimensional closed symplectic manifold admitting a Hamiltonian circle action with isolated fixed points. We show that if $M$ contains an $S^1$-invariant symplectic hypersurface $D$ such that $M\setminus D$ is a homology cell, which is satisfied when $M\setminus D$ is contractible, then $M$ and $D$ are homotopy complex projective spaces with standard Chern classes and the $S^1$-representations on the fixed-point set of $(M,D)$ are the same as those arising from the standard linear actions on $(\mathbb{P}^n,\mathbb{P}^{n-1})$, provided that $n \not \equiv 3 \pmod 4$. This can be viewed as the transformation group analogue to a recent result obtained by Peternell and the author, where the latter was conjectured by Fujita more than four decades ago.
\end{abstract}

\maketitle

\section{Introduction and main results}
Unless otherwise stated, all manifolds mentioned in this article are smooth, closed, connected, oriented, and real $2n$-dimensional. The orientation of an almost-complex manifold is the canonical one induced from its almost-complex structure. All circle actions on manifolds are nontrivial and smooth, and preserve the almost-complex structures if the manifolds in question are almost-complex. We usually denote by $S^1$, $S^1$-manifold, or $M^{S^1}$ respectively the circle, a manifold equipped with a circle action, or the fixed point set of an $S^1$-manifold $M$. A \emph{hypersurface} $D$ on a manifold $M$ is a submanifold with codimension \emph{two}.

Some topological or differential constraints on manifolds with positivity of curvature in various vague senses share similar feature to those arising from circle actions. The most prominent paradigm in this spirit perhaps is the result that the $\hat{A}$-genus of a spin manifold vanishes if it admits either a positive scalar curvature Riemannian metric or a circle action, which is due to Lichnerowicz (\cite{Lic}) and Atiyah-Hirzebruch (\cite{AH}) respectively. Hattori showed that (\cite[p.7-8]{Ha78})  some Hilbert polynomials of the $\hat{A}$-class vanish on an almost-complex $S^1$-manifold with vanishing first Betti number, while Michelsohn showed that (\cite[p.1142]{Mi} or \cite[p.365]{LM}) exactly the same conclusions hold true for K\"{a}hler manifolds with positive Ricci curvature (this is equivalent to the manifold being Fano due to the Calabi-Yau theorem), to name just a few. We remark that in the above-mentioned two examples, the main tool used in \cite{Lic} and \cite{Mi} is the Atiyah-Singer index theorem as well as the Bochner-Weitzenb\"{o}ck type formula, and that used in \cite{AH} and \cite{Ha78} is the Atiyah-Bott-Singer fixed point theorem.

One classical topic in complex differential geometry and algebraic geometry is to characterize the complex projective space $\mathbb{P}^n$ or more generally the polarized pair $\big(\mathbb{P}^n,\mathcal{O}(1)\big)$
via geometric and/or topological information as little as possible. Among various known characterizations, probably the most useful one is due to Kobayashi and Ochiai (\cite{KO}). Let $M$ be a Fano manifold and write the anti-canonical bundle $-K_M=rH$ with $r\in\mathbb{Z}_{>0}$ and $H$ an ample line bundle. The Kobayashi-Ochiai's theorem asserts that $r\leq n+1$, with equality if and only if $M=\mathbb{P}^n$ and $H$ is necessarily the hyperplane bundle. This characterization has been an indispensable criterion to show that the target manifolds or polarized pairs in various questions are $\mathbb{P}^n$ or $(\mathbb{P}^n,\mathcal{O}(1))$. For instance, it plays a key role in Siu-Yau's famous solution to the Frankel conjecture (\cite{SY}).

From the transformation group point of view, $\mathbb{P}^n$ has a typical linear circle action with $n+1$ isolated fixed points. Write $$\mathbb{P}^n=\big\{[z_1:z_2:\cdots:z_{n+1}]~|~z_i\in\mathbb{C}\big\}$$ and arbitrarily choose $n+1$ \emph{pairwise distinct} integers $a_1,a_2,\ldots,a_{n+1}$. The circle action given by
\be\label{linear action}\begin{split}
&z\cdot[z_1:z_2:\cdots:z_{n+1}]\\
:=&[z^{-a_1}z_1:z^{-a_2}z_2:
\cdots:z^{-a_{n+1}}z_{n+1}],\quad z\in S^1\subset\mathbb{C},\end{split}\ee
has $n+1$ isolated fixed points
$$P_i=[\underbrace{0:\cdots:0}_{i-1},1,0:\cdots:0],\quad 1\leq i\leq n+1.$$
Here we use the action $z^{-a_i}$ instead of $z^{a_i}$ to make the weights of the hyperplane line bundle at $P_i$ are $a_i$ (see Example \ref{example of H}).
As complex $S^1$-modules, the tangent spaces  $T_{P_i}\mathbb{P}^n$ are given by
\be\label{linear representation}T_{P_i}\mathbb{P}^n=\sum_{\overset{1\leq j\leq n+1,}{j\neq i}}t^{a_i-a_j}\in\mathbb{Z}[t,t^{-1}]=R(S^1),\quad 1\leq i\leq n+1.\ee
In other words, the weights of this circle action at $P_i$ are $\{a_i-a_j~|~j\neq i\}$.

Hattori addressed in \cite{Ha85} the question of giving a circle action analogue to the Kobayashi-Ochiai's theorem. Roughly speaking, given an almost-complex $S^1$-manifold $M$ with $n+1$ isolated fixed points, it is shown in \cite{Ha85} that if $c_1(M)=(n+1)c_1(L)$ for a specific complex line bundle $L$ over $M$, the tangent representations at these fixed points resemble (\ref{linear representation}) for some pairwise distinct integers $a_i$ arising from $L$, and $M$ is unitary cobordant to $\mathbb{P}^n$ (see Section \ref{subsection2.4} for details). It is worth mentioning that some results in \cite{Ha78} and \cite{Ha85} are also closely related to a long-standing conjecture of Petrie (\cite{Pe}). A manifold is called a \emph{homotopy complex projective space} (\emph{homotopy $\mathbb{P}^n$} for short)
if it is homotopy equivalent to $\mathbb{P}^n$. (One equivalent form of) Petrie's conjecture asserts that the Pontrjagin classes of a homotopy $\mathbb{P}^n$ equipped with a circle action must be standard. A symplectic generalization to Petrie's conjecture was addressed by Tolman (\cite{To10}).

Motivated by two basic questions of Hirzebruch concerning analytic compactification of $\mathbb{C}^n$ and the uniqueness of $\mathbb{P}^n$ (\cite[p.231]{Hi54}), Fujita proposed in \cite{Fu80} three closely related conjectures of increasing strength, which are called $A_n$, $B_n$ and $C_n$ respectively in it. Among them the weakest $A_n$ asserts that a K\"{a}hler manifold $M$ containing a smooth hypersurface $D$ such that $M\setminus D$ is contractible must be $\mathbb{P}^n$ with $D$ a hyperplane $\mathbb{P}^{n-1}$. The important subcase of $M\setminus D\cong\mathbb{C}^n$, which was also conjectured in \cite[p.151]{BM}, has been solved by C. Li and Zhou in \cite{LZ} (see also \cite{Pet} for the case of even $n$). When $n \not \equiv 3 \pmod 4$, conjecture $A_n$ was recently solved by Peternell and the author (\cite{LP}).
\begin{definition}\label{homology cell} Let $D$ be a hypersurface of a $2n$-dimensional manifold $M$. $M \setminus D$ is called a \emph{homology cell} if the natural homomorphisms $$H_k(D;\mathbb{Z})\longrightarrow H_k(M;\mathbb{Z})$$ are bijective for all $0\leq k\leq 2(n-1)$.
\end{definition}
\begin{remark}\label{definition remark}
It is well-known that $M\setminus D$ being a homology cell implies that the natural homomorphisms $$H^{k}(M;\mathbb{Z})\longrightarrow H^{k}(D;\mathbb{Z})$$ are also bijective for $0\leq k\leq 2(n-1)$ (\cite[Cor.3.4]{Hat}).
\end{remark}

It turns out that the contractibility of $M\setminus D$ implies that it is a homology cell (see Lemma \ref{homo trivial implies homo cell}). Indeed what we solved in \cite{LP} is conjecture $B_n$ in \cite{Fu80} for $n \not \equiv 3 \pmod 4$, which asserts that a K\"{a}hler manifold $M$ containing a smooth hypersurface $D$ such that $M\setminus D$ is a homology cell must be $\mathbb{P}^n$ with $D$ a hyperplane on it. Note that the main tools in \cite{LZ} are of complex and symplectic geometry nature and those in \cite{LP} are of characteristic classes nature, and hence are different. But both proofs rely on the the above-mentioned Kobayashi-Ochiai's criterion.
The hardest conjecture $C_n$, which asserts that a Fano manifold with the same integral cohomology ring as $\mathbb{P}^n$ must be $\mathbb{P}^n$, is only known when $n\leq6$. We refer the reader to \cite[\S 1]{LP} and the references therein for more background on these materials.

Note that with respect to the linear circle action (\ref{linear action}) on $\mathbb{P}^n$, the hyperplane
\be\label{hyperplane}\mathbb{P}^{n-1}=\big\{[z_1:\cdots:z_{n+1}]
\in\mathbb{P}^n~|~z_{n+1}=0\big\}\ee
is $S^1$-invariant with $n$ fixed points $P_1,\ldots,P_{n}$, and as complex $S^1$-modules,
\be\label{linear representation2}T_{P_i}\mathbb{P}^{n-1}=
\sum_{\overset{1\leq j\leq n,}{j\neq i}}t^{a_i-a_j},\quad 1\leq i\leq n.\ee
Moreover, the circle actions on both $\mathbb{P}^n$ and its hyperplane $\mathbb{P}^{n-1}$ are Hamiltonian with respect to their standard symplectic forms induced from their Fubini-Study metrics (see Section \ref{subsection2.3} for more details).

With the principle that similar constraints may be shared by the conditions of positivity curvature and circle action in mind, it is natural to ask whether or not there exists a circle action counterpart to the above-mentioned results in \cite{LP}. \emph{The main purpose} of this article is to prove the following result, which can be regarded as the circle action analogue to the main result in \cite{LP}, where the latter is a solution to conjecture $B_n$ in \cite{Fu80} whenever the dimensions $n \not \equiv 3 \pmod 4$.
\begin{theorem}\label{main theorem}
Let $M$ be a symplectic manifold admitting a Hamiltonian circle action with isolated fixed points. If $M$ contains an $S^1$-invariant symplectic hypersurface $D$ such that $M\setminus D$ is a homology cell. Then
\begin{enumerate}
\item
$M$ and $D$ are both homotopy complex projective spaces.

\item
When $n \not \equiv 3 \pmod 4$, $M$ and $D$ have standard Chern classes and the $S^1$-representations on their fixed points are of the forms (\ref{linear representation}) and (\ref{linear representation2}) respectively for some pairwise distinct integers $a_1,\ldots,a_{n+1}$.

\item
When $n\equiv 3 \pmod 4$, the conclusions above are still true provided that
$$\big(c_1(M),c_1(D)\big)\neq\big(\frac{1}{2}(n+1)x,
\frac{1}{2}(n-1)x_0\big),$$ where $x$ and $x_0$ are generators of $H^2(M;\mathbb{Z})$ and $H^2(D;\mathbb{Z})$ respectively such that $x^n[M]=x_0^{n-1}[D]=1$ and $x|_{D}=x_0$ \rm(see Remark \ref{definition remark}\rm).
\end{enumerate}
\end{theorem}
\begin{remark}\label{remark after main theorem}
\begin{enumerate}
\item
$M$ and $D$ in Theorem \ref{main theorem} have \emph{no} odd-dimensional cohomology (see Lemma \ref{Morse function}). Thus the condition of $M\setminus D$ being a homology cell is equivalent to that the natural homomorphisms $H_{2k}(D;\mathbb{Z})\longrightarrow H_{2k}(M;\mathbb{Z})$ are all bijective for $0\leq k\leq n-1$.

\item
We suspect that the case $$\big(c_1(M),c_1(D)\big)=\big(\frac{1}{2}(n+1)x,
\frac{1}{2}(n-1)x_0\big)$$ may not happen even if $n\equiv 3 \pmod 4$. But our methods are not able to rule out its possibility (see the proof in Proposition \ref{first Chern class}).
\end{enumerate}
\end{remark}

\begin{definition}
A topological space $X$ is called \emph{homologically trivial} if the reduced homology $\widetilde{H}_i(X;\mathbb{Z})=0$ for all $i$, which is satisfied when $X$ is contractible.
\end{definition}
Indeed the homological triviality of $M\setminus D$ is enough to derive that it is a cohomology cell (see Lemma \ref{homo trivial implies homo cell}). Thus we have the following consequence, which is a circle action analogue to \cite[Thm 1.7]{LP}, where the latter is a solution to conjecture $A_n$ in \cite{Fu80} whenever $n \not \equiv 3 \pmod 4$.
\begin{corollary}\label{main corollary}
Let $M$ be a symplectic manifold admitting a Hamiltonian circle action with isolated fixed points. If $M$ contains an $S^1$-invariant symplectic hypersurface $D$ such that $M\setminus D$ is homologically trivial, the same conclusions as in Theorem \ref{main theorem} hold. In particular, this holds true when $M\setminus D$ is contractible.
\end{corollary}

The main ideas in the proof of Theorem \ref{main theorem} consist of two parts. The first part is to further develop some ideas in \cite{LP}, which focus on K\"{a}hler manifolds, to the setting of symplectic manifolds. The second part is to carefully verify that in our setting the technical assumptions in \cite{Ha85} be satisfied. To this end, a wide range of tools need to be employed to provide a proof.

The rest of this article is organized as follows. We briefly review some basic notation and facts on circle action needed in this article in Section \ref{section2}.
Some ideas in \cite{LP}, which focus on K\"{a}hler manifolds, shall be further developed in Section \ref{section3} to the setting of symplectic manifolds equipped with Hamiltonian circle action with isolated fixed points. Section \ref{section4} is then devoted to verifying that in our model the assumptions in \cite{Ha85} are satisfied. The whole proof of our main results will be finished in Section \ref{section5}.

\section{Preliminaries}\label{section2}
Some basic background materials will be recalled in this section for later use as well as for completeness.

\subsection{Almost-complex $S^1$-manifolds}\label{subsection2.1}
We assume in this subsection that $M=(M,J)$ be an almost-complex $S^1$-manifold whose fixed points are nonempty and isolated, and $M^{S^1}=\{P_1,\ldots,P_m\}$, where as is well-known the number $m$ of these isolated fixed points is exactly the Euler characteristic of $M$.

Let $T_{P_i}M=(T_{P_i}M,J)$ be the tangent spaces to $P_i$, which are $n$-dimensional complex vector spaces. The circle action on $M$ induces a linear action of $S^1$ on these $T_{P_i}M$, which means that these $T_{P_i}M$ can be viewed as complex $S^1$-modules. Since these fixed points $P_i$ are isolated, only zero vector in $T_{P_i}M$ can be fixed by the whole $S^1$. In other words, there exist \emph{nonzero} integers $k_1^{(i)},\ldots,k_n^{(i)}$ such that
$$T_{P_i}M=\sum_{j=1}^nt^{k_j^{(i)}}\in\mathbb{Z}[t,t^{-1}]=R(S^1),\quad 1\leq i\leq m,$$
and these $k_1^{(i)},\ldots,k_n^{(i)}$ are called the \emph{weights} of this circle action at $P_i$.

If $D$ is an $S^1$-invariant almost-complex submanifold on $M$, then $D^{S^1}\subset M^{S^1}$. If $D^{S^1}$ is nonempty and some $P_i\in D^{S^1}$, then $T_{P_i}D$ is a complex $S^1$-submodule of $T_{P_i}M$ and the weights at $P_i$ of the restricted circle action on $D$ are those $k_j^{(i)}$ whose one-dimensional submodules span $T_{P_i}D$.

The residue formula of Bott (\cite{Bo2}), in the case of almost-complex $S^1$-manifolds, allows us to calculate Chern numbers of $M$ in terms of the weights $k_j^{(i)}$ around the fixed points $P_i$. In our later proof only the formula for the Chern number $c_1^n[M]$ is needed, and so we record it below for later reference.
\begin{lemma}\label{Bott residue formula}
The residue formula for the Chern number $c_1^n[M]$ reads as follows \rm(see \cite[\S2.1]{LL}\rm).
\begin{eqnarray}\label{Bott residue formula0}
\sum_{i=1}^{m}\frac{(\sum_{j=1}^mk_j^{(i)})^r}
{\prod_{j=1}^{n}k_j^{(i)}}=\left\{ \begin{array}{ll}
0,\quad& 0\leq r\leq n-1,\\
~\\
c_1^n[M],\quad& r=n.
\end{array} \right.
\end{eqnarray}
\end{lemma}
\begin{remark}
Note that the formula (\ref{Bott residue formula0}) \emph{not only} provides a method for calculating the Chern number $c_1^n[M]$ ($r=n$), but also gives many restrictions to the weights $k_j^{(i)}$ ($0\leq r\leq n-1$). For some related applications of these restrictions to $S^1$-manifolds we refer the interested reader to \cite{LL} and \cite{Li12b}, and the references therein.
\end{remark}

\subsection{Admissible complex line bundles}\label{subsection2.2}
\begin{definition}\label{admissible}
A (complex) line bundle $L$ over an almost-complex $S^1$-manifold $M$ is called \emph{admissible} if the $S^1$-action on $M$ can be lifted to $L$ making it an $S^1$-bundle.
\end{definition}
It turns out that (\cite{HY}) $L$ is admissible if and only if its first Chern class $c_1(L)$ lies in the image of the natural homomorphism $$H^2_{S^1}(M;\mathbb{Z})\longrightarrow H^2(M;\mathbb{Z}),$$
where $H^2_{S^1}(M;\mathbb{Z})$ is the equivariant cohomology group. A simple but useful fact is that \emph{any} line bundle on an $S^1$-manifold with vanishing first Betti number is admissible (see \cite[Lemma 3.1]{Li12b}).

If $L$ is an admissible line bundle over an almost-complex $S^1$-manifold $M$ with isolated fixed points $P_i$, we may fix a lifting of this action to it. Then $L_{P_i}$, the fibers of $L$ at $P_i$, are complex one-dimensional $S^1$-modules and so there exist integers $a_i$ such that
$$L_{P_i}=t^{a_i}\in\mathbb{Z}[t,t^{-1}]=R(S^1).$$
These integers $a_i$ are called \emph{weights} of $L$ at the fixed points $P_i$ with respect to this given lifting. However, if we choose \emph{another} lifting of this action to $L$, the weights at $P_i$ are changed simultaneously to $a_i+a$ for some integer $a$. So the weights $\{a_i\}$ of an admissible $L$ are only determined up to a simultaneous integer. In particular, the integers $a_i-a_j$ make sense.

\subsection{Hamiltonian circle action}\label{subsection2.3}
In this subsection we assume that $M=(M,\omega)$ be a symplectic manifold.

A circle action on $M$ is called \emph{symplectic} if it preserves the symplectic form $\omega$. If a circle acts on $(M,\omega)$ symplectically, it is well-known that we can always find an almost-complex structure both compatible with $\omega$ and preserved by this circle action. So the notions in Section \ref{subsection2.1} are applicable to the setting of symplectic circle actions \emph{without} explicitly mentioning this compatible almost-complex structure.

Let $X$ be the generating vector field of a circle action on $(M,\omega)$. A circle acts on $(M,\omega)$ symplectically if and only if the one form $\omega(X,\cdot)$ is closed, which is due to the Cartan formula. A symplectic circle action is called \emph{Hamiltonian} if the one form $\omega(X,\cdot)$ is exact, i.e., $\omega(X,\cdot)=\text{d}f$ for some smooth function $f$ on $M$. This $f$ is called the \emph{moment map} of this Hamiltonian circle action, which is unique up to an additive constant. Note that the fixed-point set of a Hamiltonian circle action is exactly the critical-point set of the moment map $f$, and hence is always nonempty as the points minimizing or maximizing $f$ are always critical. The linear action (\ref{linear action}) on $(\mathbb{P}^n,\omega_{0})$ is a Hamiltonian circle action with respect to the symplectic form $\omega_{0}$ induced from the Fubini-Study metric.

Let $(M,\omega)$ admit a Hamiltonian circle action whose generating vector field and moment map are $X$ and $f$, and $D$ its $S^1$-invariant symplectic submanifold. Then the generating vector field of this $S^1$-manifold $D$ is $X|_{D}$ and the restriction of $\omega(X,\cdot)=\text{d}f$ to $D$ implies that the restricted circle action on $D$ is also Hamiltonian with the moment map $f|_{D}$. Moreover, we have $D^{S^1}\subset M^{S^1}$. The hyperplane $\mathbb{P}^{n-1}$ in (\ref{hyperplane}) is an $S^1$-invariant symplectic hypersurface of $(\mathbb{P}^n,\omega_0)$.

For later reference, some well-known facts concerning Hamiltonian
circle action with isolated fixed points are collected in the following lemma.
\begin{lemma}\label{Morse function}
Let $(M,\omega)$ be a symplectic manifold equipped with a Hamiltonian circle action with isolated fixed points. Then we have the following facts.
\begin{enumerate}
\item
$M$ is simply-connected, and the homology $H_{\ast}(M;\mathbb{Z})$ is torsion-free and has no (nontrivial) odd-dimensional elements.

\item
Each even-dimensional Betti number $b_{2p}(M)$ \rm($0\leq p\leq n$\rm) is equal to the number of isolated fixed points whose weights have exactly $p$ negative integers.
\end{enumerate}
In particular, if $D$ is an $S^1$-invariant symplectic submanifold of such $(M,\omega)$, these conclusions also hold true for $D$.
\end{lemma}
\begin{proof}
These are consequences of the classical Morse theory. Indeed,
the moment map $f:M\rightarrow\mathbb{R}$ of this Hamiltonian circle action is a \emph{perfect Morse function} whose critical points are precisely the isolated fixed points of this circle action. Moreover, the Morse index of each point is \emph{twice} the number of the negative weights at it. When the manifold in question is K\"{a}hler, these results are due to Frankel (\cite{Fr59}) building on Bott's work \cite{Bo}. For general symplectic manifolds, these facts can be found in \cite{Ki}.
\end{proof}

\subsection{A circle action analogue to Kobayashi-Ochiai's theorem}\label{subsection2.4}
Let $M=(M,J)$ be an almost-complex $S^1$-manifold with $M^{S^1}=\{P_1,\ldots,P_m\}$, and the symbols and notions introduced in Section \ref{subsection2.1} are kept using.

The following notion of quasi-ampleness for line bundles is due to Hattori, which is treated in \cite{Ha85} as a circle action analogue to ample line bundles in algebraic geometry.
\begin{definition}\label{def of quasi-ample}
Let $L$ be an admissible line bundle over $M$ with weights $a_i$ at $P_i$ under some lifting. $L$ is called \emph{quasi-ample} if the weights $a_1,\ldots,a_m$ are pairwise distinct and $c_1^n(L)[M]\neq0$, where $[M]$ is the fundamental class of $M$.
\end{definition}\label{quasi-ample}
\begin{remark}\label{remark quasi-ample}
\begin{enumerate}
\item
As mentioned at the end of Section \ref{subsection2.2}, this definition is independent of the choice of the lifting as another lifting makes the weights $a_i$ simultaneously to $a_i+a$ for some integer $a$.

\item
By definition $L$ is quasi-ample if and only if $L^{-1}$ is, which is crucially different from the fact in algebraic geometry.
\end{enumerate}
\end{remark}

This notion can be simply illustrated by the linear action on $\mathbb{P}^n$ mentioned in the Introduction.
\begin{example}\label{example of H}
Let $H$ be the hyperplane bundle on $\mathbb{P}^n$.
The total space of the tautological bundle $H^{-1}$ is given by
 $$\big\{([z_1:\cdots:z_{n+1}],v)~\big|~v\in\mathbb{C}
 (z_1,\ldots,z_{n+1})\subset\mathbb{C}^{n+1}\big\}.$$
The circle action given by (\ref{linear action}) can be lifted to the tautological bundle via $$z\cdot\big([z_1:\cdots:z_{n+1}],v\big):=\big([z^{-a_1}z_1:
\cdots:z^{-a_{n+1}}z_{n+1}],v\big).$$
whose weights at the fixed points are $-a_i$. Thus the weights of this induced lift on $H$ are $a_i$. Moreover, $c_1^n(H)[\mathbb{P}^n]=1$ and $c_1^n(H^{-1})[\mathbb{P}^n]=(-1)^n$. So $H$ as well as $H^{-1}$ is quasi-ample.
\end{example}

The following notion is also due to Hattori (\cite[p.447]{Ha85}).
\begin{definition}\label{condition C}
We say an admissible line bundle $L$ over $M$ satisfies the \emph{condition $C$} if there exist integers $k_0\in\mathbb{Z}_{\geq0}$ and $a\in\mathbb{Z}$ such that
\be\label{condition C formula}
\sum_{j=1}^nk_{j}^{(i)}=k_0a_i+a,\quad\forall ~1\leq i\leq m,\ee
where $k_{j}^{(i)}$ are the weights of the circle action at $P_i$ and $a_i$ the weights of $L$ with respect to some lifting of this circle action.
\end{definition}
\begin{remark}
\begin{enumerate}
\item
With the same reason as in Remark \ref{remark quasi-ample}, this definition is independent of the choice of the lifting.

\item
The condition (\ref{condition C formula}) is called ``Condition $D$" by Hattori in \cite{Ha85}. Here we call it Condition $C$ as the letter ``$D$" has been reserved to denote a hypersurface in $M$.
\end{enumerate}
\end{remark}

With these notions understood, the following are the main results in \cite{Ha85} (\cite[Thms 5.1, 5.7]{Ha85}), which are Hattori's solution to the circle action analogue to the Kobayashi-Ochiai's theorem.
\begin{theorem}[Hattori]\label{Hattori}
Let $M$ be an almost-complex $S^1$-manifolds with $M^{S^1}=\{P_1,\ldots,P_m\}$, and $L$ a quasi-ample line bundle over it satisfying the Condition $C$ in (\ref{condition C formula}) for some $k_0\in\mathbb{Z}_{\geq0}$ and $a\in\mathbb{Z}$. Then $k_0\leq n+1\leq m$. If moreover $k_0=n+1=m$, then $M$ is unitary cobordant to $\mathbb{P}^n$, $c_1^n(L)[M]=1$, and $$\big\{k_1^{(i)},\ldots,k_n^{(i)}\big\}=
\big\{a_i-a_j~|~j\neq i,~1\leq j\leq n+1\big\},\quad\forall~1\leq i\leq n+1.$$
In other words, the $S^1$-representations at these $P_i$ are the same as (\ref{linear representation}).
\end{theorem}
\begin{remark}
Hattori conjectured that the assumptions involved in Theorem \ref{Hattori} can be further weakened and some precise conjectures are formulated in \cite{Ha86}, which, to the author's best knowledge, are widely open.
\end{remark}

\section{Homotopy complex projective spaces with circle actions}\label{section3}
The main purpose of this section is to develop, in parallel, some ideas in \cite{LP} to the symplectic setting, which culminates in Proposition \ref{first Chern class}.

\subsection{Homotopy complex projective spaces}
First we explain in this subsection that the pair $M$ and $D$ in Theorem \ref{main theorem} are both homotopy complex projective spaces, and the homological triviality of $M\setminus D$ implies that it is a homology cell. Their proofs are essentially scattered in \cite{LP}.
\begin{lemma}\label{homotopy Pn}
The manifolds $M$ and $D$ in Theorem \ref{main theorem} are both homotopy complex projective spaces.
\end{lemma}
\begin{proof}
It can be shown that if $M\setminus D$ is a homology cell with the first Betti number $b_1(M)=0$, then the integral cohomology rings of $M$ and $D$ are the same as those of $\mathbb{P}^n$ and $\mathbb{P}^{n-1}$ respectively. Such manifolds are usually called \emph{cohomology complex projective spaces}. This fact is (implicitly) due to Sommese (\cite[p.64-65]{So}) and Fujita (\cite[Thm 2]{Fu80}). A detailed proof can be found in \cite[Lemma 2.1]{LP}. Lemma \ref{Morse function} tells us that $b_1(M)=0$ and hence $M$ and $D$ in Theorem \ref{main theorem} are cohomology complex projective spaces. Moreover, it is well-known that a \emph{simply-connected} cohomology $\mathbb{P}^n$ must be a homotopy $\mathbb{P}^n$ (\cite[Prop.4.1]{LP}). Thus this lemma follows, still from Lemma \ref{Morse function}.
\end{proof}

\begin{lemma}\label{homo trivial implies homo cell}
If $M\setminus D$ is  homologically trivial, then it is a homology cell. In particular, Corollary \ref{main corollary} follows from Theorem \ref{main theorem}.
\end{lemma}
\begin{proof}
The proof is a simple application of the Poincar\'{e}-Alexander-Lefschetz duality theorem. Indeed, for compact subsets $B\subset A$ in $M$, we have (\cite[p.351]{Br})
$$\label{PD}H_{k}(A,B;\mathbb{Z})\cong H^{2n-k}(M-B,M-A;\mathbb{Z}),\quad 0\leq k\leq 2n.$$
Taking $(A,B)=(M,D)$ in above leads to
$$H_k(M,D;\mathbb{Z})\cong H^{2n-k}(M-D;\mathbb{Z})=0,\quad0\leq k\leq 2n-1,$$
as $M\setminus D$ is homologically trivial. Then the homology long exact sequence for the pair $(M,D)$ reads
$$\cdots\longrightarrow H_{k+1}(M,D;\mathbb{Z})\longrightarrow H_{k}(D;\mathbb{Z})\overset{}{\longrightarrow}
H_{k}(M;\mathbb{Z})\longrightarrow H_k(M,D;\mathbb{Z})\longrightarrow\cdots.$$
Hence the natural homomorphisms $$H_k(D;\mathbb{Z})\overset{}{\longrightarrow}H_k(M;\mathbb{Z})$$ are bijective for
$0\leq k\leq 2(n-1)$.
\end{proof}

\subsection{The Chern number $c_1c_{n-1}$}
The main goal of this subsection is to show that the Chern number $c_1c_{n-1}[M]$ (resp. $c_1c_{n-2}[D]$) in Theorem \ref{main theorem} is the same as that of $\mathbb{P}^n$ (resp. $\mathbb{P}^{n-1}$). To this end, we first recall the notion of the \emph{$\chi_y$-genus} for almost-complex manifolds, which was first introduced by Hirzebruch (for projective manifolds) in \cite{Hi66}.

Let $M=(M,J)$ be an almost-complex manifold with almost-complex structure $J$. This $J$ induces the usual $\bar{\partial}$-operator acting on $\Omega^{p,q}(M)$ ($0\leq p,q\leq n$), the space of smooth complex-valued $(p,q)$-forms: $$\Omega^{p,q}(M)\xrightarrow{\bar{\partial}}\Omega^{p,q+1}(M).$$ The choice of
an almost-Hermitian metric on $M$ enables us to define the formal adjoint
$\bar{\partial}^{\ast}$ of the $\bar{\partial}$-operator: $$\Omega^{p,q}(M)\xrightarrow{\bar{\partial}^{\ast}}\Omega^{p,q-1}(M).$$

For every $0\leq p\leq n$, we have
the following Dolbeault-type elliptic operator $D_p$:
\be\label{GDC}D_p:~\bigoplus_{\textrm{$q$
even}}\Omega^{p,q}(M)\xrightarrow{\bar{\partial}+\bar{\partial}^{\ast}}\bigoplus_{\textrm{$q$
odd}}\Omega^{p,q}(M),\ee whose index is denoted by $\chi^{p}(M)$ in
the notation of Hirzebruch.
When $J$ is integrable, i.e., $M$ is a complex manifold, it is known that  $$\chi^{p}(M)=\sum_{q=0}^n(-1)^qh^{p,q}(M),$$
 where $h^{p,q}(M)$ are the Hodge numbers of $M$.
The \emph{Hirzebruch's
$\chi_{y}$-genus} of $M$, denoted by $\chi_{y}(M)$, is the generating
function of these $\chi^p(M)$:
$$\chi_{y}(M):=\sum_{p=0}^{n}\chi^{p}(M)y^{p}.$$
For example, we have $\chi_y(\mathbb{P}^n)=\sum_{i=0}^n(-y)^i$.

The general form of the Hirzebruch-Riemann-Roch theorem, which was first established by Hirzebruch for projective manifolds (\cite{Hi66}) and in the general case by Atiyah and Singer (\cite{AS}), allows us to compute these indices $\chi^{p}(M)$ (and hence $\chi_y(M)$) in terms of Chern numbers of $M$. To be more precise, we have (\cite[\S5.4]{HBJ})
\be\label{HRR}\chi_y(M)=\int_M
\prod_{i=1}^n\frac{x_i(1+ye^{-x_i})}{1-e^{-x_i}},\ee
where $x_1,\ldots,x_n$ are formal Chern roots of $M$, i.e., the
$i$-th elementary symmetric polynomial of $x_1,\ldots,x_n$ represents $c_i(M)$, the
$i$-th Chern class of $M$.

By (\ref{HRR}) it is well-known that $\chi_y(M)\big|_{y=-1}=c_n[M]$, which is the Euler characteristic of $M$ and the constant term of the Taylor expansion of $\chi_y(M)$ at $y=-1$. The following explicit formula of the coefficient in front of $(y+1)^2$ in the Taylor expansion of $\chi_y(M)$ has been observed independently in several articles (\cite[p.18]{NR}, \cite[p.141-143]{LW}, \cite[Thm 4.1]{Sa}).
\begin{lemma}\label{Taylor expansion}
If we set
\be\label{chiy-1}\int_M\prod_{i=1}^n\frac{x_i(1+ye^{-x_i})}{1-e^{-x_i}}
=:\sum_{j=0}^nK_j(M)\cdot(y+1)^j,\ee
then
$$K_2(M)=
\frac{1}{12}\Big[\frac{n(3n-5)}{2}c_{n}+c_{1}c_{n-1}\Big][M].$$
In particular, the Chern number $c_1c_{n-1}$ of an almost-complex manifold is completely determined by its $\chi_y$-genus.
\end{lemma}
\begin{remark}
For basic properties of the coefficients $K_j$ in (\ref{chiy-1}) and their related applications in \cite{NR}, \cite{LW} and \cite{Sa} as well as in others, we refer the interested reader to \cite[\S3.2]{Li19} for a detailed summary. This kind of thing turns out to be a special case of what we called \emph{-1-phenomenon} in \cite{Li15} and \cite{Li17}.
\end{remark}

Now we come to the following fact.
\begin{lemma}\label{Chern number}
Let $M$ be a symplectic manifold endowed with a Hamiltonian circle action with isolated fixed points. If the Betti numbers $b_i(M)=b_i(\mathbb{P}^n)$ for all $i$, then $$c_1c_{n-1}[M]=c_1c_{n-1}[\mathbb{P}^n]=\frac{1}{2}n(n+1)^2.$$ In particular, the Chern numbers $c_1c_{n-1}[M]$ and $c_1c_{n-2}[D]$ in Theorem \ref{main theorem} are respectively the same as those of $\mathbb{P}^n$ and $\mathbb{P}^{n-1}$ due to Lemma \ref{homotopy Pn}.
\end{lemma}
\begin{proof}
Since $b_i(M)=b_i(\mathbb{P}^n)$, the Euler characteristic $c_n[M]=n+1$. Following the notation in Section \ref{subsection2.1}, let $M^{S^1}=\{P_1,\ldots,P_{n+1}\}$ and the weights around $P_i$ are $k_1^{(i)},\ldots,k_n^{(i)}$. The Atiyah-Bott-Singer fixed point theorem allows us to compute $\chi_y(M)$ in terms of the information around $M^{S^1}$. In our case this formula reads (\cite[(2.6)]{Li12a}):
\be\label{localizaition of chiy}
\chi_y(M)=\sum_{i=1}^{n+1}(-y)^{d_i}=
\sum_{i=1}^{n+1}(-y)^{n-d_i},
\ee
where $d_i$ is the number of negative integers in $k_1^{(i)},\ldots,k_n^{(i)}$. Lemma \ref{Morse function}, together with the facts $b_i(M)=b_i(\mathbb{P}^n)$, implies that $\{d_1,\ldots,d_{n+1}\}=\{0,1,\ldots,n\}.$ Thus (\ref{localizaition of chiy}) becomes
$$
\chi_y(M)=\sum_{i=0}^{n}(-y)^{i},$$
which is equal to $\chi_y(\mathbb{P}^n)$. Therefore
the desired result follows from Lemma \ref{Taylor expansion}.
\end{proof}
\begin{remark}
\begin{enumerate}
\item
This kind of residue-type formula (\ref{localizaition of chiy}) was first obtained by Kosniowski in the setting of complex manifolds equipped with holomorphic vector fields (\cite[Thm 1]{Ko70}). In \cite{Li12a} the author applied this idea to give applications to symplectic manifolds endowed with symplectic or Hamiltonian circle actions with isolated fixed points.

\item
For K\"{a}hler manifolds with the same Betti numbers as $\mathbb{P}^n$, Lemma \ref{Chern number} was observed in \cite[Cor.2.5]{LW}, which turns out to be very useful in the characterization of $\mathbb{P}^n$ as the first Chern class is involved in it.
\end{enumerate}
\end{remark}

\subsection{The first Chern class}
The goal in this subsection is to determine the first Chern classes $c_1(M)$ and $c_1(D)$ in Theorem \ref{main theorem}.

Let $D\overset{i}{\hookrightarrow}M$ be the inclusion map of the invariant hypersurface $D$ in $M$ in Theorem \ref{main theorem}, and $N_D$ the complex line bundle over $D$ normal to the complex tangent bundle $(T_D,J)$ of $D$ in the complex tangent bundle $(T_M,J)$ of $M$. Since the canonical orientations induced from the almost-complex structures are understood, by Lemma \ref{homotopy Pn} and Remark \ref{definition remark}, we shall choose once for all the generators $x\in H^2(M;\mathbb{Z})$ and $x_0\in H^2(D;\mathbb{Z})$ such that
\be\label{cohomology data}
x^n[M]=x_0^{n-1}[D]=1,\quad i^{\ast}(x)=x_0,\quad c_1(N_D)=x_0.
\ee

With all materials in this section understood, we can now deduce the following result, whose arguments are basically parallel to those scattered in \cite{LP}.
\begin{proposition}\label{first Chern class}
Let $M$ and $D$ be as in Theorem \ref{main theorem}. Then $$\big(c_1(M),c_1(D)\big)=\text{$\big((n+1)x,nx_0\big)$ or $\big(\frac{1}{2}(n+1)x,\frac{1}{2}(n-1)x_0\big)$},$$
where the latter may occur only if $n\equiv 3\pmod 4$.
\end{proposition}
\begin{proof}
The three complex vector bundles $(T_M,J)$, $(T_D,J)$ and $N_D$ are related by the short exact sequence
\be\label{short exact sequence}0\longrightarrow (T_D,J)\longrightarrow i^{\ast}(T_M,J)\longrightarrow N_D\longrightarrow0.\ee
The Chern classes $$c_i(M)=c_i(T_M,J)\in\mathbb{Z}x^i,\quad c_j(D)=c_j(T_D,J)\in\mathbb{Z}x_0^i
\quad(1\leq i\leq n,~1\leq j\leq n-1),$$
both due to Lemma \ref{homotopy Pn}. 
By abuse of notation these $c_i(M)$ and $c_j(D)$ can be viewed as \emph{integers} by ignoring $x^i$ and $x_0^j$ respectively. 

With this convention understood, taking $c_1(i^{\ast}(T_M,J))$ and $c_{n-1}(i^{\ast}(T_M,J))$ in the short exact sequence (\ref{short exact sequence}) and applying the facts in (\ref{cohomology data}) yield
\be\label{1}c_1(M)=c_1(D)+1,\quad c_{n-1}(M)=n+c_{n-2}(D),\ee
where the fact that $c_{n-1}(D)=n$ is the Euler characteristic of $D$ has been used.

Lemma \ref{Chern number} implies that
\be\label{2}c_1(M)c_{n-1}(M)=\frac{1}{2}n(n+1)^2,\quad c_1(D)c_{n-2}(D)=\frac{1}{2}(n-1)n^2.\ee
Combing (\ref{1}) with (\ref{2}) leads to
\be\label{3}\frac{n(n+1)^2}{2c_1(M)}=n+\frac{(n-1)n^2}{2(c_1(M)-1)}.\ee
Solving (\ref{3}) gives us $c_1(M)=n+1$ or $\frac{1}{2}(n+1)$.

It suffices to show that the case $c_1(M)=\frac{1}{2}(n+1)$ may occur only if $n\equiv 3\pmod 4$. Indeed, $M$ is homotopy equivalent to $\mathbb{P}^n$ by Lemma \ref{homotopy Pn}. Assume that $c_1(M)=\frac{1}{2}(n+1)$. Since $c_1(M)$ modulo two, which is the second Stiefel-Whitney class of $M$, is a homotopy type invariant due to Wu's formula (\cite[p.130]{MS}), thus
$$\frac{1}{2}(n+1)\equiv n+1 \pmod{2},$$
which deduces that $n\equiv 3\pmod 4$.
\end{proof}

\section{Verification of the assumptions in Theorem \ref{Hattori}}\label{section4}
In this section the technical assumptions in Theorem \ref{Hattori} shall be verified to hold true in our setting.

We assume in this section that $M$ be as in Theorem \ref{main theorem}. The notation and symbols introduced in Sections \ref{section2} and \ref{section3} are kept used in the sequel.
\begin{lemma}\label{condition n+1}
Let $L$ be the line bundle over $M$ such that $c_1(L)=x$, which is admissible. If $c_1(M)=(n+1)x$, then the Condition C in (\ref{condition C formula}) holds for $L$ with $k_0=n+1$, i.e.,
\be\label{condition C formula n+1}
\sum_{j=1}^nk_{j}^{(i)}=(n+1)a_i+a,\quad\forall ~1\leq i\leq n+1,\ee
where $a$ is some integer and $a_1,\ldots,a_{n+1}$ are the weights of $L$ at the fixed points $P_i$ with respect to some lifting of this circle action.
\end{lemma}
\begin{proof}
Any line bundle is admissible since $b_1(M)=0$ (\cite{HY}, \cite[Lemma 3.1]{Li12b}). Arbitrarily choose a lifting of $L$ and let $a_1,\ldots,a_{n+1}$ be the weights of $L$ at the fixed points $P_1,\ldots,P_{n+1}$ with respect to it.

Let the line bundle $L_0:=\bigwedge^n(TM,J)$, the $n$-the wedge of the complex bundle $(TM,J)$, which is usually called anti-canonical line bundle in algebraic geometry. Then $c_1(L_0)=c_1(M)$ and its weights at $P_i$ with respect to the natural action induced from that on $M$ are $\sum_{j=1}^{n}{k_j^{(i)}}$.

The condition of $c_1(M)=(n+1)x$ implies that the two line bundles $L_0$ and $L^{\otimes(n+1)}$ are equal in the isomorphic sense: $L_0=L^{\otimes(n+1)}$. Note that the weights of $L^{\otimes(n+1)}$ at $P_i$ with respect to the above-chosen lifting are $(n+1)a_i$. This means that, with respect to two possibly different liftings, the weights of the line bundle $L_0=L^{\otimes(n+1)}$ at $P_i$ are respectively $\sum_{j=1}^{n}{k_j^{(i)}}$ and $(n+1)a_i$. Therefore for all $i$ they must differ by a simultaneous integer, say $a$. This shows the desired (\ref{condition C formula n+1}) and completes the proof.
\end{proof}

With Lemma \ref{condition n+1} in hand, in order to apply Theorem \ref{Hattori}, it suffices to show that the line bundle $L$ in Lemma \ref{condition n+1} is quasi-ample.

\begin{lemma}\label{quasi-ample lemma}
The line bundle $L$ in Lemma \ref{condition n+1} is quasi-ample.
\end{lemma}
\begin{proof}
By Definition \ref{quasi-ample} it suffices to show that the weights $a_i$ are pairwise distinct as $c_1^n(L)[M]=1$, which, according to (\ref{condition C formula n+1}), is equivalent to showing that the $n+1$ integers $$\sum_{j=1}^nk_{j}^{(1)},\ldots,\sum_{j=1}^nk_{j}^{(n+1)}$$
are pairwise distinct. We shall apply Lemma \ref{Bott residue formula} to derive this conclusion.

For simplicity let
$$\lambda_i:=\sum_{j=1}^nk_{j}^{(i)},\quad\text{and}\quad e_i:=\prod_{j=1}^nk_{j}^{(i)},\quad 1\leq i\leq n+1.$$
Lemma \ref{Bott residue formula} says that
\begin{eqnarray}\label{Bott residue formula1}
\sum_{i=1}^{n+1}\frac{\lambda_i^r}
{e_i}=\left\{ \begin{array}{ll}
0,\quad& 0\leq r\leq n-1,\\
~\\
c_1^n[M],\quad& r=n.
\end{array} \right.
\end{eqnarray}

What we want to show is that the $n+1$ integers $\lambda_1,\ldots,\lambda_{n+1}$ are pairwise distinct. To this end, we list $\lambda_1,\ldots,\lambda_{n+1}$, among which \emph{a priori} some integers may be the same, in increasing order as
$$\widetilde{\lambda}_1<\widetilde{\lambda}_2<
\cdots<\widetilde{\lambda}_t,\quad (1\leq t\leq n+1).$$
So we need to show that $t=n+1$. Let
\be\label{sum}\mu_i:=\sum_{\overset{1\leq j\leq n+1,}{\lambda_j=\widetilde{\lambda}_i}}\frac{1}{e_j},\quad 1\leq i\leq t.\ee
With this notation in mind the formula (\ref{Bott residue formula1}) can be restated as
\begin{eqnarray}\label{Bott residue formula2}
\sum_{i=1}^{t}(\widetilde{\lambda}_i)^r\cdot\mu_i
=\left\{ \begin{array}{ll}
0,\quad& 0\leq r\leq n-1,\\
~\\
c_1^n[M],\quad& r=n.
\end{array} \right.
\end{eqnarray}

Suppose on the contrary that $t\leq n$. Then (\ref{Bott residue formula2}) implies that
\begin{eqnarray}\label{Bott residue formula3}
\left\{ \begin{array}{ll}
\mu_1+\mu_2+\cdots+\mu_t=0\\
~\\
\widetilde{\lambda}_1\cdot\mu_1+
\widetilde{\lambda}_2\cdot\mu_2+
\cdots+\widetilde{\lambda}_t\cdot\mu_t=0\\
\qquad\vdots\\
~\\
(\widetilde{\lambda}_1)^{t-1}\cdot\mu_1+
(\widetilde{\lambda}_2)^{t-1}\cdot\mu_2+
\cdots+(\widetilde{\lambda}_t)^{t-1}\cdot\mu_t=0.
\end{array} \right.
\end{eqnarray}
The Vandermonde determinant $$\det\big((\widetilde{\lambda}_i)^{j-1})\big)_{1\leq i,j\leq t}\neq 0$$ 
as these $\widetilde{\lambda}_i$ are pairwise distinct. This implies that the only solution to (\ref{Bott residue formula3}) is $$\mu_1=\cdots=\mu_t=0,$$
which in turn implies that $c_1^n[M]=0$ due to the case of $r=n$ in (\ref{Bott residue formula2}). This contradicts to the fact that $c_1^n[M]=(n+1)^n\neq0$. Therefore the assumption $t\leq n$ is impossible, which gives the desired result.
\end{proof}
\begin{remark}
When $t\leq n$, we show that all $\mu_i=0$. This implies that for every $i$ the sum on the right-hand side of (\ref{sum}) has at least two terms. So what we really show is that when the $S^1$-manifold $M$ has $n+1$ fixed points, either the $n+1$ integers $\lambda_1,\ldots,\lambda_{n+1}$ are pairwise distinct, or for any $\lambda_i$, there exists some $j\neq i$ such that $\lambda_i=\lambda_j$.
\end{remark}

\section{Proof of Theorem \ref{main theorem} and Corollary \ref{main corollary}}\label{section5}
Let $M$ and $D$ be as in Theorem \ref{main theorem}, and as assumed therein, either $n\not\equiv 3\pmod 4$, or $n\equiv 3 \pmod 4$ but $(c_1(M),c_1(D))\neq(\frac{1}{2}(n+1)x,\frac{1}{2}(n-1)x_0)$.

The reason that Corollary \ref{main corollary} follow from Theorem \ref{main theorem} has been explained in Lemma \ref{homo trivial implies homo cell}. We now finish the proof of Theorem \ref{main theorem}.
The fact that $M$ and $D$ be homotopy complex projective spaces has been shown in Lemma \ref{homotopy Pn}.

\begin{lemma}
The $S^1$-modules $T_{P_i}M$ \rm($1\leq i\leq n+1$\rm) are the same
as (\ref{linear representation}) for some pairwise integers $a_1,\ldots,a_{n+1}$.
\end{lemma}
\begin{proof}
By Proposition \ref{first Chern class} we have $c_1(M)=(n+1)x$. Lemmas \ref{condition n+1} and \ref{quasi-ample lemma} imply that the line bundle $L$ over $M$ is quasi-ample and satisfies the Condition C in (\ref{condition C formula}) with $k_0=n+1=m$. Then Theorem \ref{Hattori} tells us that the weights $k_j^{(i)}$ at $P_i$ are of the form
\be\label{4}\big\{k_1^{(i)},\ldots,k_n^{(i)}\big\}=
\big\{a_i-a_j~|~j\neq i,~1\leq j\leq n+1\big\}
,\quad1\leq i\leq n+1,\ee
where these pairwise integers $a_1,\ldots,a_{n+1}$ are weights of the quasi-ample $L$ at $P_i$ with respect to some lifting of the circle action on $M$. This means that the complex $S^1$-modules $T_{P_i}M$ are given by
$$T_{P_i}M=\sum_{\overset{1\leq j\leq n+1,}{j\neq i}}t^{a_i-a_j},\quad 1\leq i\leq n+1,$$
which are isomorphic to those $T_{P_i}\mathbb{P}^n$ in (\ref{linear representation}).
\end{proof}

The $S^1$-invariant symplectic hypersurface $D$ is also a symplectic manifold endowed with the Hamiltonian circle action induced from that on $M$ with $n$ isolated fixed points. Since $D^{S^1}\subset M^{S^1}$, we may assume without loss of generality that $D^{S^1}=\{P_1,\ldots,P_n\}$.
\begin{lemma}\label{lemma D}
The $S^1$-modules $T_{P_i}D$ \rm($1\leq i\leq n$\rm) are the same
as (\ref{linear representation2}) for the above-mentioned integers $a_1,\ldots,a_{n}$.
\end{lemma}
\begin{proof}
The restriction of $L$ to $D$, $L_D$, is also quasi-ample over $D$, because, with respect to the same lifting above, the weights of $L_D$ at $P_i$ are the $n$ pairwise integers $a_1,\ldots,a_n$.

For each $1\leq i\leq n$, the $S^1$-module $T_{P_i}D$ is a submodule of $T_{P_i}M$ and therefore the $n-1$ weights of the $S^1$-module $T_{P_i}D$ belong to the following $n$-element set
\be\label{5}\big\{a_i-a_j~|~j\neq i,~1\leq j\leq n+1\big\}\ee
due to (\ref{4}).

Note that $c_1(D)=nx_0$ due to Proposition \ref{first Chern class} and $c_1(L_D)=x_0$, which imply that $c_1(D)=nc_1(L_D)$. The same arguments as in the proof of Lemma \ref{condition C formula n+1} deduce that, for each $1\leq i\leq n$, the sum of the $(n-1)$ weights of the $S^1$-module $T_{P_i}D$ and $na_i$ differs by a simultaneous integer. Since the $(n-1)$ weights come from the $n$-element set (\ref{5}), the only possibility is that the $(n-1)$ weights at $T_{P_i}D$ are $$\{a_i-a_j~|~j\neq i,~1\leq j\leq n\},\quad 1\leq i\leq n,$$ 
and the simultaneous difference from $na_i$ is $-(a_1+\ldots+a_n)$, as easily checked. This means that
$$T_{P_i}D=\sum_{\overset{1\leq j\leq n,}{j\neq i}}t^{a_i-a_j},\quad 1\leq i\leq n,$$
which is the same as $T_{P_i}\mathbb{P}^{n-1}$ in (\ref{linear representation2}).
\end{proof}

\begin{lemma}
The Chern classes of $M$ and $D$ are the same as those of $\mathbb{P}^n$ and $\mathbb{P}^{n-1}$ respectively.
\end{lemma}
\begin{proof}
Since $M$ is unitary cobordant to $\mathbb{P}^n$ by Theorem \ref{Hattori}, the Chern numbers of $M$ are the same as those of $\mathbb{P}^{n}$. Note that $c_1(M)=(n+1)x$, hence for each $2\leq i\leq n$ we have
$$c_ic_1^{n-i}[M]=c_ic_1^{n-i}[\mathbb{P}^n]=
\binom{n+1}{i}(n+1)^{n-i},$$
which implies that $c_i(M)=\binom{n+1}{i}x^i$. This completes the proof for $M$. The same arguments as above give the desired proof for $D$.
\end{proof}

\end{document}